\documentclass[a4paper,11pt]{amsart}

\pdfoutput=1

\usepackage[text={400pt,660pt},centering]{geometry}

\usepackage{amsthm, amssymb, amsmath, amsfonts, mathrsfs, bbm}
\usepackage{mathtools}
\usepackage{scalerel} 
\usepackage{stackrel}
\usepackage{comment}
\usepackage{a4wide}
\usepackage{float}

\usepackage[colorlinks=true, pdfstartview=FitV, linkcolor=blue, citecolor=blue, urlcolor=blue,pagebackref=false]{hyperref}
\usepackage{esint} 
\usepackage{MnSymbol} 

\usepackage{microtype}
\usepackage{cleveref}

\newtheorem{proposition}{Proposition}
\newtheorem{theorem}[proposition]{Theorem}
\newtheorem{lemma}[proposition]{Lemma}

\theoremstyle{remark}
\newtheorem{remark}[proposition]{Remark}

\theoremstyle{definition}

\numberwithin{equation}{section}
\numberwithin{proposition}{section}
\numberwithin{table}{section}

\renewcommand{\leq}{\leqslant}
\renewcommand{\geq}{\geqslant}

\renewcommand{\epsilon}{\varepsilon}

\title[Non-coincidence of critical points for directed
polymers on percolation]{Non-coincidence of critical points for directed
polymers on supercritical percolation clusters}

\author[F.\ Cottini]{
Francesca Cottini}
\address{Laboratoire de Probabilit{\'e}, Statistique et Mod{\'e}lisation of the Sorbonne Universit{\'e}, Sorbonne Universit{\'e}}
\curraddr{4 place Jussieu, 75005 Paris, France}
\email{francesca.cottini@sorbonne-universite.fr}

\author[M.\ Nitzschner]{
Maximilian Nitzschner}
\address{Department of Mathematics, The Hong Kong University of Science and Technology}
\curraddr{Clear Water Bay, Kowloon, Hong Kong}
\email{mnitzschner@ust.hk}

\date{\today}

\begin{document}

\begin{abstract}

We consider the model of a directed polymer in a random environment defined on the infinite cluster of supercritical Bernoulli bond percolation in dimensions $d \geq 3$. For this model, it was proved in~\cite{nitzschner2025absence} that for almost every realization of the cluster, the polymer is in a \textit{strong disorder} regime for any positive inverse temperature. Here, we show for almost every realization of the cluster the existence of a non-empty sub-phase of the strong disorder regime, consisting of positive inverse temperatures in which \textit{very strong disorder} does not hold. This is in contrast to the recently established sharpness of the phase transition for the directed polymer on the full lattice, see~\cite{junk2024strong,junk2025coincidence}.

\bigskip


\medskip


\end{abstract}
\maketitle

\tableofcontents

%
%
%
%
%
%
%
%
\section{Introduction}

We investigate the model of a directed polymer in a random, i.i.d.~environment, defined on a typical realization of the infinite cluster of Bernoulli (bond) percolation on $\mathbb{Z}^d$ with $d \geq 3$, in the supercritical phase.  \medskip

Since its inception in the eighties, there has been considerable effort to understand the behavior of the directed polymer model in random environment, primarily on $\mathbb{Z}^d$, $d \geq 1$, and we refer to the monograph~\cite{comets2017directed} and the recent review~\cite{zygouras2024directed} for an overview. A prominent feature of this model is the emergence of a \textit{phase transition}, which occurs as a real parameter $\beta \geq 0$ (the inverse temperature) varies. The different phases of the model can be characterized by the asymptotic decay (or lack thereof) of the normalized partition function $W_n^\beta$, see~\eqref{eq:Normalized-Partition-function}, as the time horizon $n$ tends to infinity. Two critical parameters $\beta_c \leq \overline{\beta}_c$ have been introduced to separate  different phases of the model: in the \textit{weak disorder} regime $\beta \in [0,\beta_c)$, the limit of the normalized partition function (as $n  \to \infty$) is almost surely positive, in the \textit{strong disorder} regime $\beta \in (\beta_c,\infty)$ it is zero almost surely, and in the \textit{very strong disorder} regime $\beta \in (\overline{\beta}_c,\infty)$ the decay of $(W_n^\beta)_{n \geq 1}$ is exponentially fast in $n$ (thus by definition $\beta_c \leq \overline{\beta}_c$). On $\mathbb{Z}^d$, $d \in \{1,2\}$, one knows that $\overline{\beta}_c = 0$, see~\cite{comets2005majorizing} and~\cite{lacoin2010new} for $d = 1$ and $d = 2$, respectively. On $\mathbb{Z}^d$,  $d\geq 3$, the model fulfills $\beta_c > 0$ (see~\cite{bolthausen1989note}), and the question of \textit{sharpness}, meaning whether $\beta_c = \overline{\beta}_c$, was only recently solved positively in~\cite{junk2024strong} for environments  with a law unbounded from above, and in~\cite{junk2025coincidence} for more general environments (and certain long-range walks, but still on $\mathbb{Z}^d$, $d\geq 3$). In the latter two references, it was also proved that $\beta_c$ itself belongs to the weak disorder regime. \medskip

Several new phenomena emerge for the directed polymer model if the underlying simple random walk is considered on more general locally finite and connected graphs. For specific graphs (including trees, complete graphs, or hierarchical lattices), the analysis is sometimes simplified compared to the full lattice, see, e.g.,~\cite{MR3908906,MR3692316,
brunet2000probability,
buffet1993directed,MR4294278,
MR4531010,
comets2019random,cook-derrida,
derrida1989directed,derrida1988polymers,
eckmann1989largest,MR2594367}, but the understanding of the model is much more limited for general graphs with an irregular (potentially random) structure. A systematic approach to the model in a general framework was initiated in~\cite{cosco2021directed} (see also~\cite{kajino2020two}). In~\cite{nitzschner2025absence}, motivated by a question in~\cite{cosco2021directed}, it is proved that for a polymer in a random environment defined on a typical realization $\mathcal{C}_\infty$ of an infinite cluster in Bernoulli (bond) percolation on $\mathbb{Z}^d$, $d  \geq 3$, for $p \in (p_c(d),1)$ (with $p_c(d)$ the percolation threshold), one has $\beta_c(\mathcal{C}_\infty) = 0$, unlike on the full lattice. The main contribution of the present work is to show that, for any typical realization of $\mathcal{C}_\infty$ in $d\geq 3$, there is a (non-empty) subset of positive inverse temperatures for which the polymer is in a strong disorder regime, but very strong disorder does \textit{not} hold. In other words, $0 = \beta_c(\mathcal{C}_\infty) < \overline{\beta}_c(\mathcal{C}_\infty)$, which is in contrast to the sharpness of the phase transition on the full lattice (see~\cite{junk2024strong,
junk2025coincidence}). This behavior has been previously observed for random walks with very heavy-tailed jumps on $\mathbb{Z}^d$, see~\cite{viveros2023directed} and for biased random walks on certain supercritical Galton-Watson trees, see~\cite[Section 6]{cosco2021directed}, but to our knowledge appears to be new for nearest-neighbor walks on a random graph which closely resembles $\mathbb{Z}^d$, $d \geq 3$, on large scales. \medskip

We now describe our model and results in more detail. Throughout the article we tacitly assume that $d\geq 3$. We consider for $p \in [0,1]$ the probability measure $\mathbb{Q}_p$ on the space $\{0,1\}^{\mathbb{E}^d}$, where $\mathbb{E}^d$ denotes the set of nearest-neighbor edges in the Euclidean lattice, such that the canonical coordinates $\mu_e : \{0,1\}^{\mathbb{E}^d} \rightarrow \{0,1\}$, $e \in \mathbb{E}^d$, are i.i.d.~Bernoulli random variables with parameter $p$. Throughout the article, we fix 
\begin{equation}
\label{eq:Critical-p}
p \in (p_c(d),1), \qquad \text{where } p_c(d) = \inf\{p \in [0,1] \, : \, \mathbb{Q}_p[0 \leftrightarrow \infty ] > 0 \} (\in (0,1)),
\end{equation}
(where the event under $\mathbb{Q}_p$ stands for the existence of an unbounded nearest-neighbor path of edges with $\mu_e = 1$, starting from the origin). We denote by $\mathcal{C}_\infty (= \mathcal{C}_\infty(\mu))$ the infinite subgraph of $(\mathbb{Z}^d, \{e \in \mathbb{E}^d \, : \, \mu_e = 1 \})$ with minimal Euclidean distance to the origin, which exists and is fact $\mathbb{Q}_p$-a.s.~unique, see~\cite{grimmett1999percolation}. It is often convenient to center the cluster at the origin by considering the probability measure
\begin{equation}
\label{eq:Q0Def}
\widehat{\mathbb{Q}}_p = \mathbb{Q}_p\left[ \cdot \, | \, 0 \in \mathcal{C}_\infty \right].
\end{equation}
We are interested in results for a fixed, typical (under $\widehat{\mathbb{Q}}_p$) realization of $\mu \in \{0 \in \mathcal{C}_\infty\}$, and we consider for such a $\mu$ and $z \in \mathcal{C}_\infty$:
\begin{equation}
\begin{minipage}{0.8\linewidth}
$P_{z,\mu} =$ the canonical probability measure on $(\mathbb{Z}^d)^{\mathbb{N}_0}$, equipped with the cylinder $\sigma$-algebra generated by the canonical projections $(X_k)_{k \geq 0}$, such that the latter is a discrete-time simple random walk on $\mathcal{C}_\infty$ starting from $z$  
\end{minipage}
\end{equation}
(see Section~\ref{s.notation} for a precise definition). We then introduce the random environment by considering 
\begin{equation}
\label{eq:Definitions-of-omega}
\begin{minipage}{0.8\linewidth}
$\omega = (\omega(n,z)))_{n \in \mathbb{N},z\in \mathbb{Z}^d}$ i.i.d.~real-valued random variables under a probability measure $\mathbb{P}$, fulfilling the assumptions 
$$
\mathbb{E}[\omega(1,0)] = 0, \ \ \mathbb{E}[\omega(1,0)^2] = 1, \ \  \lambda(\beta) = \log \mathbb{E}[e^{\beta \omega(1,0)}] < \infty \text{ for }\beta \geq -a,
$$
for some constant $a > 0$, with $\mathbb{E}$ denoting the expectation with respect to $\mathbb{P}$.
\end{minipage}
\end{equation}
For a given percolation configuration $\mu \in \{0 \in \mathcal{C}_\infty\}$ and environment $\omega \in \mathbb{R}^{\mathbb{N} \times \mathbb{Z}^d}$, we then define the \textit{normalized partition function}
\begin{equation}
\label{eq:Normalized-Partition-function}
W_{n,\mu}^\beta(\omega) = E_{0,\mu}\left[\exp\left\{\beta \sum_{i = 1}^n \omega(i,X_i) - n\lambda(\beta) \right\}  \right],
\end{equation}
(with $E_{0,\mu}$ denoting the expectation under the probability measure $P_{0,\mu}$) of time horizon $n \in \mathbb{N}$, at inverse temperature $\beta \geq 0$. The \textit{directed polymer measure} at time horizon $n \in \mathbb{N}$ is then defined as a measure on nearest-neighbor paths $X = (X_k)_{k \geq 0}$ on $\mathbb{Z}^d$ (in fact, all trajectories are in $\mathcal{C}_\infty$, $P_{0,\mu}$-a.s.) by setting
\begin{equation}
\mathrm{d}P_{0,\mu}^{n,\beta}[X] = \frac{1}{W_{n,\mu}^\beta(\omega)} \exp\left\{\beta \sum_{i = 1}^n \omega(i,X_i) - n\lambda(\beta) \right\} \mathrm{d}P_{0,\mu}[X],
\end{equation}
which favors paths with larger values of $\omega$ and with $\beta \geq 0$ characterizing the disorder strength. The phase transition in $\beta$ between delocalization or localization can in fact be characterized by the asymptotic behavior of $W_{n,\mu}^\beta$. 
As observed in~\cite{bolthausen1989note} for the corresponding model on $\mathbb{Z}^d$ (see~\cite[Section 1.2]{cosco2021directed} for the set-up of general graphs), $(W_{n,\mu}^{\beta})_{n \geq 1}$ is a non-negative martingale with respect to the filtration $(\mathcal{F}_n)_{n \geq 1}$, given by $\mathcal{F}_n = \sigma(\omega(i,x) \, : \, i \leq n, x \in \mathbb{Z}^d)$, fulfilling
\begin{equation}
\mathbb{E}[W_{n,\mu}^{\beta}] = 1, \qquad \text{for all }n \in \mathbb{N}.
\end{equation}
By the standard martingale convergence theorem, the sequence $(W^\beta_{n,\mu})_{n \geq 1}$ converges $\mathbb{P}$-a.s.~to a limit $W_{\infty,\mu}^\beta$. If $W_{\infty,\mu}^\beta > 0$ holds $\mathbb{P}$-a.s., we say that the polymer is in the \textit{weak disorder regime}, whereas if $W_{\infty,\mu}^\beta = 0$ holds $\mathbb{P}$-a.s., the polymer is said to be in the \textit{strong disorder regime} (these are the only possibilities by a $0$-$1$-law, see~\cite[Proposition 1.2]{cosco2021directed}). One further has a monotonicity property, again by~\cite[Proposition 1.2]{cosco2021directed}, stating that if strong disorder holds for a given $\beta$, it holds for any $\beta' > \beta$, and one can therefore define for any fixed $\mu \in \{0 \in \mathcal{C}_\infty\}$ the critical inverse temperature
\begin{equation}
\beta_c(\mathcal{C}_\infty) = \inf\{ \beta \in [0,\infty) \, : \, \text{strong disorder holds at $\beta$} \} \in [0,+\infty]
\end{equation}
(with the convention $\inf \varnothing = +\infty$). To characterize a potentially strict sub-phase of the strong disorder regime, we define for fixed $\mu \in \{0 \in \mathcal{C}_\infty\}$ the \textit{quenched free energy} as 
\begin{equation}
\label{eq:Quenched-Free-Energy}
\mathfrak{f}(\beta,\mu) = \limsup_{n \rightarrow \infty} \frac{1}{n} \mathbb{E}\Big[ \log W^{\beta}_{n,\mu} \Big]=  \limsup_{n \rightarrow \infty}  \frac{1}{n}\log W_{n,\mu}^\beta,
\end{equation}
where the second equality holds $\mathbb{P}$-a.s., see~\cite[Proposition 1.8]{cosco2021directed}. By the same reference, one can define for a given $\mu \in \{0 \in \mathcal{C}_\infty\}$ the critical parameter
\begin{equation}
(\beta_c(\mathcal{C}_\infty) \leq ) \ \overline{\beta}_{c}(\mathcal{C}_\infty) = \inf\{\beta \in [0,\infty) \, : \, \mathfrak{f}(\beta,\mu) < 0 \} \in [0,+\infty],
\end{equation}
and we say that \textit{very strong disorder} holds at $\beta \in [0,\infty)$ if $\mathfrak{f}(\beta,\mu)< 0$, corresponding to an exponential decay of $W_{n,\beta}^\mu$ to zero, so 
\begin{equation}
\begin{cases}
\text{for $\beta > \overline{\beta}_c(\mathcal{C}_\infty)$, }& \mathfrak{f}(\beta,\mu)< 0, \\
\text{for $\beta < \overline{\beta}_c(\mathcal{C}_\infty)$, }& \mathfrak{f}(\beta,\mu) = 0.
\end{cases}
\end{equation}
 By~\cite[Theorem 1.1]{nitzschner2025absence}, one knows that the directed polymer on $\mathcal{C}_\infty$ is always in a strong disorder regime for positive $\beta$:
\begin{equation}
\label{eq:Main-result-previous}
\begin{minipage}{0.8\linewidth}
let $d \geq 3$ and $p\in (p_c(d),1)$, then for $\widehat{\mathbb{Q}}_p$-a.e.~realization of $\mu \in \{0 \in \mathcal{C}_\infty \}$, one has $\beta_c(\mathcal{C}_\infty) = 0$.
\end{minipage}
\end{equation}
Moreover, it follows from~\cite[Proposition 1.9]{cosco2021directed}, that very strong disorder holds for large enough $\beta$ (which may depend on $\mu$), for $\widehat{\mathbb{Q}}_p$-a.e.~$\mu \in \{0 \in \mathcal{C}_\infty\}$, also assuming that the law of $\omega(1,0)$ is unbounded from above. On the percolation cluster, this leaves open the question whether for a typical realization of $\mu\in \{0 \in \mathcal{C}_\infty\}$, one has $\overline{\beta}_c(\mathcal{C}_\infty) > 0$, see~\cite[Remark 4.4 (3)]{nitzschner2025absence}. Our main result answers this question affirmatively. 
\begin{theorem}
\label{thm:MainTheorem}
Let $d \geq 3$, $p \in (p_c(d),1)$, and assume that the law of $\omega(1,0)$ under $\mathbb{P}$ is unbounded from above. Then, for $\widehat{\mathbb{Q}}_p$-a.e.~realization of $\mu$ and the infinite cluster $\mathcal{C}_\infty$:
\begin{enumerate}
\item[(i)] one has 
\begin{equation}
\label{eq:Main-Claim-i}
\overline{\beta}_c(\mathcal{C}_\infty) = \overline{\beta}^{\,\mathrm{cluster}}_c
\end{equation}
for a deterministic constant $\overline{\beta}^{\, \mathrm{cluster}}_c \in [0,\infty)$;
\item[(ii)] the following deterministic lower bound holds: 
\begin{equation}
\label{eq:Main-result-bound}
\overline{\beta}^{\,\mathrm{cluster}}_c \geq \beta_{L^2}(\mathbb{Z}^d) = \sup\left\{\beta \geq 0  \, : \, \Big(e^{\lambda(2\beta) - 2\lambda(\beta)} -1 \Big)\sum_{k = 1}^\infty (P^{\mathbb{Z}^d}_0)^{\otimes 2}[X_k^{(1)} = X_k^{(2)}] < 1 \right\} (> 0),
\end{equation}
(with $X^{(1)}, X^{(2)}$ two independent simple random walks on $\mathbb{Z}^d$, $d \geq 3$, starting from the origin, under the measure $(P^{\mathbb{Z}^d}_0)^{\otimes 2}$).
\end{enumerate}
\end{theorem}
In particular, this shows that the phase transition is \textit{not} sharp, differently from what happens on the full lattice. \medskip

We briefly put this result in context. As mentioned above, the question of \textit{sharpness} of the phase transition of the directed polymer model on $\mathbb{Z}^d$, $d \geq 3$, namely the fact that
\begin{equation}
\beta_c(\mathbb{Z}^d) = \overline{\beta}_c(\mathbb{Z}^d) (> 0),
\end{equation} 
when the underlying random walk is simple
was recently resolved in~\cite{junk2024strong} and~\cite{junk2025coincidence}. For $d = 1$ and $d = 2$, $\beta_c(\mathbb{Z}^d) = 0$ was proved in~\cite{carmona2002partition} and~\cite{comets2003directed}, respectively. The fact that $\overline{\beta}_c(\mathbb{Z}^d) = 0$, for $d =1$ and $d = 2$ appeared in~\cite{comets2005majorizing} and~\cite{lacoin2010new}, respectively, implying the sharpness of the model also in $\mathbb{Z}^d$, $d \in \{1,2\}$. As remarked in~\cite[Section 2.5]{junk2024strong}, this sharpness result also holds for any transitive lattice with polynomial growth. Previously known cases in which $\beta_c < \overline{\beta}_c$ occurs include directed polymers for which the underlying random walk is either defined on $\mathbb{Z}^d$ but has very heavy tailed jumps (see~\cite{viveros2023directed}, see also~\cite[Proposition 2.11]{junk2025coincidence} for an example showing that $\beta_c = 0$ and $\overline{\beta}_c = \infty$ is possible), or corresponds to an upwardly biased (positive recurrent) random walk on a Galton-Watson tree with some additional assumptions on the offspring distribution (see~\cite[Theorem 6.1]{cosco2021directed}). Thus, to our knowledge the result obtained here seems to be new for an example of a nearest-neighbor random walk on a transient graph.
\medskip

Theorem~\ref{thm:MainTheorem} shows in combination with the main result of~\cite{nitzschner2025absence}, cf.~\eqref{eq:Main-result-previous}, that for $\widehat{\mathbb{Q}}_p$-almost every realization $\mu \in \{0 \in \mathcal{C}_\infty\}$ of the cluster, there is a  regime $\beta \in (0,\overline{\beta}^{\,\mathrm{cluster}}_c)$, in which $\mathfrak{f}(\beta,\mu) = 0$, meaning that $ W_{n,\mu}^{\beta} \to 0$ as $n \to \infty$, $\mathbb{P}$-a.s., but the decay happens at a smaller-than-exponential rate. This raises the question how fast this decay occurs. To that end, we give a partial answer as follows.
\begin{theorem}
\label{thm:Decay-rate}
Let $d \geq 3$, $p \in (p_c(d),1)$. There exists a constant $\kappa = \kappa(p) > 0$ such that for $\widehat{\mathbb{Q}}_p$-a.e.~$\mu \in \{0 \in \mathcal{C}_\infty\}$ with corresponding infinite cluster $\mathcal{C}_\infty$, if $\beta \in (0,\overline{\beta}^{\,\mathrm{cluster}}_c)$, $\mathbb{P}$-a.s.,
\begin{equation}
\label{eq:Bound-main-result}
\lim_{n \rightarrow \infty} \frac{\log^{\kappa}(n)}{n} \log W_{n,\mu}^\beta = -\infty.
\end{equation}
\end{theorem}
This shows that for $\beta \in (0,\overline{\beta}^{\,\mathrm{cluster}}_c)$, the decay rate is almost exponential. The parameter $\kappa > 0$ could be made somewhat more explicit by carefully tracking certain exponential tail estimates related to the stochastic integrability of the relative cluster volume in large boxes, see~\eqref{eq:Exponent} below, see also Remark~\ref{rem:Final-Remarks}, but we do not expect our control on $\kappa$ to be sharp. \medskip

We give a brief outline of the proof of Theorem~\ref{thm:MainTheorem}. Part (i) follows essentially directly from the ergodicity of the shift operators and the fact that the critical inverse temperature does not depend on the starting point of the walk on the cluster, as was proved in~\cite{cosco2021directed}. For part (ii), we utilize a similar approach as in~\cite{nitzschner2025absence},
showing that there are sufficiently many \textit{good regions} attached to $\mathcal{C}_\infty$ that are present within a box of side-length $n$ (for $n \in \mathbb{N}$ large enough), and argue that these regions are visited by the underlying walk with a sizeable probability. In the present set-up, these regions will be chosen as ``solid $\mathbb{Z}^d$-like blocks'', of size $(\log n)^{\frac{1}{d}}$ (as opposed to tubes of size $\log n$ used in~\cite{nitzschner2025absence}, or also~\cite{abe2015effective} for a different question). We then use a coupling between two independent random walks on $\mathbb{Z}^d$ to prove that $\log W_{n,\mu}^\beta$ cannot decay linearly with a good probability, together with a classical concentration bound from~\cite{liu2009exponential} to obtain the claim. This strategy is inspired by the one employed in~\cite{cosco2021directed} for the case of biased random walks on certain Galton-Watson trees.
  \medskip

This article is organized as follows: in Section~\ref{s.notation}, further notation is introduced together with some useful known results on random walks, as well as a concentration inequality for the normalized partition function taken from~\cite{liu2009exponential} that we need in the subsequent sections. In the short Section~\ref{s.Deterministic}, we  prove part (i) of Theorem~\ref{thm:MainTheorem}. In Section~\ref{s.Boxes-percolation}, we prove a quantitative bound guaranteeing the existence of aforementioned ``good blocks''. In Section~\ref{s.ProofMainProp} the proof of part (ii) of the main result Theorem~\ref{thm:MainTheorem} is presented. Finally, we prove Theorem~\ref{thm:Decay-rate} concerning the decay of $W^{\beta}_{n,\mu}$ in Section~\ref{sec:Decay}.  \medskip

We will use the following convention concerning constants. We denote by $C,c,c',...$ positive constants depending only on the dimension $d$ which may change from place to place. Numbered constants $c_1,c_2,...$ refer to the value assigned to them at their first appearance in the text. Dependence on any additional parameters will be explicit in the notation.

%
%
%
%
%
%
%
%
\section{Notation and useful results}
\label{s.notation}

In this section further notation is introduced, together with some pertinent results for the simple random walk on $\mathbb{Z}^d$, $d \geq 3$, as well as on the infinite cluster of supercritical Bernoulli percolation. We also state a concentration bound on the logarithm of the normalized partition function from~\cite{liu2009exponential}, which will be useful in the proof of our main result in Section~\ref{s.ProofMainProp}. \medskip

We now introduce some further notation. Throughout the article, we will assume that $d \geq 3$, unless stated otherwise. We denote the set of non-negative integers by $\mathbb{N}_0 = \{0,1,2,...\}$ and the set of positive integers by $\mathbb{N} = \mathbb{N}_0 \setminus \{0\}$. It will be convenient to introduce the sub-lattice $\mathbb{Z}^d_o \subseteq \mathbb{Z}^d$ consisting of points $x \in \mathbb{Z}^d$ such that the sum of the coordinates $(x_i)_{1 \leq i \leq d}$ is odd, and the sub-lattice $\mathbb{Z}^d_e = \mathbb{Z}^d \setminus \mathbb{Z}^d_o$ for the points $x \in \mathbb{Z}^d$ such that the sum of the coordinates $(x_i)_{1 \leq i \leq d}$ is even. For real numbers $a,b$, we write $a \vee b$ and $a \wedge b$ for the maximum and minimum between $a$ and $b$, respectively, and let $\lfloor a \rfloor$ stand for the integer part of $a$. We denote by $| \, \cdot \, |$, $| \, \cdot \, |_1$, and $| \, \cdot \, |_\infty$ the Euclidean, $\ell^1$- and $\ell^\infty$-norms on $\mathbb{R}^d$, respectively. We call vertices $x,y \in \mathbb{Z}^d$ nearest neighbors if $|x-y| = 1$ and write $x \sim y$ in this case. The set of edges in $\mathbb{Z}^d$ is defined as $\mathbb{E}^d = \{\{x,y\} \, : \, x,y\in \mathbb{Z}^d, x \sim y \}$. For a given set $A \subseteq \mathbb{Z}^d$, we denote by $E(A)$ the set of edges $\{x,y\} \in \mathbb{E}^d$ with $A \cap \{x,y\} \neq \emptyset$. For $x \in \mathbb{Z}^d$ and $R \geq 0$, the closed ball (in $\ell^\infty$-norm) with center $x$ and radius $R$ is defined as $B(x,R) = \{y \in \mathbb{Z}^d \, : \, |x-y|_\infty \leq R\} \subseteq \mathbb{Z}^d$. We also write $B_o(x,R) = B(x,R) \cap \mathbb{Z}^d_o$ and $B_e(x,R) = B(x,R) \cap \mathbb{Z}^d_e$ for the points in $B(x,R)$ with odd resp.~even parity. Given two non-empty sets $A, B \subseteq \mathbb{Z}^d$, we write $d_\infty(A,B) = \inf\{|x-y|_\infty \, : \, x \in A, y \in B\}$ and abbreviate $d_\infty(\{x\},B)$ as $d_\infty(x,B)$ for $x \in \mathbb{Z}^d$. Similarly, we define $d_{1}(A,B) = \inf\{|x-y|_1 \, : \, x \in A, y \in B\}$ and use the same convention for $d_1(x,B)$. \medskip

We now state some simple facts about the discrete-time simple random walk on $\mathbb{Z}^d$. Let $(X_n)_{n \geq 0}$ denote the canonical process on $(\mathbb{Z}^d)^{\mathbb{N}_0}$ and let $P_x^{\mathbb{Z}^d}$ stand for the canonical law of a simple random walk on $\mathbb{Z}^d$ starting from $x \in \mathbb{Z}^d$. For $A \subseteq \mathbb{Z}^d$, we let $T_A = \inf\{n \in \mathbb{N}_0 \, : \, X_n \notin A\}$ stand for the exit time of the walk from $A$ (with the convention $\inf \varnothing = +\infty$). We also write 
\begin{equation}
p_n^A(x,y) = P_x^{\mathbb{Z}^d}[X_n = y, T_A > n], \qquad x,y \in \mathbb{Z}^d, n \in \mathbb{N}_0,
\end{equation}
for the transition probability of the walk killed upon exiting $A$. We will need the following classical Gaussian lower bound, taken from~\cite[Theorem 4.25]{barlow2017random}:
\begin{equation}
\label{eq:Killed-HK-bound}
\begin{minipage}{0.8\linewidth}
for $x_0 \in \mathbb{Z}^d$, $R \geq 1$, $x,y \in B(x_0,R/2)$, and $|x-y|_1 \leq n \leq R^2$, one has
$$
p_n^{B(x_0,R)}(x,y) + p_{n+1}^{B(x_0,R)}(x,y) \geq \frac{C}{n^{\frac{d}{2}}}\exp\left(-c \frac{|x-y|_1^2}{n} \right).
$$
\end{minipage}
\end{equation}
The following standard estimate on the confinement probability of a simple random walk in a box will be useful in Section~\ref{s.ProofMainProp}. Its proof is included for completeness.
\begin{lemma}
\label{lem:Confinement}
For any $n \in \mathbb{N}$ and integer $R \geq 10$, one has
\begin{equation}
\label{eq:Confinement-bound}
P_0^{\mathbb{Z}^d}[T_{B(0,R)} \geq n] \geq \exp\left(- c \frac{n}{R^2} \right).
\end{equation}
\end{lemma}
\begin{proof}
By applying~\eqref{eq:Killed-HK-bound}, we see that for any $x \in B(0,R/2)$, one has
\begin{equation}
P_x^{\mathbb{Z}^d}[X_{R^2} \in B(0,R/2), T_{B(0,R)} > R^2] \geq \sum_{y \in B(0,R/2)} \frac{C}{R^d} \inf_{y' \in B(0,R/2)} \exp\left(-c \frac{|x-y'|_1^2}{R^2} \right) \geq c.
\end{equation}
By applying the strong Markov property $\lfloor n/R^2 \rfloor$ times we see that $P_{0}^{\mathbb{Z}^d}[T_{B(0,R)} \geq n] \geq c^{\lfloor n/R^2 \rfloor}$, hence the claim follows.
\end{proof}

Next, we turn to the simple random walk on $\mathcal{C}_\infty$. We equip the space $\{0,1\}^{\mathbb{E}^d}$ with the canonical $\sigma$-algebra $\mathcal{A}$ generated by the coordinate projections $(\mu_e)_{e \in \mathbb{E}^d}$ on $\{0,1\}^{\mathbb{E}^d}$, and recall the notation $\mathbb{Q}_p$ for the probability measure on $(\{0,1\}^{\mathbb{E}^d}, \mathcal{A})$ governing Bernoulli percolation with parameter $p \in [0,1]$ from the introduction, as well as the conditional measure $\widehat{\mathbb{Q}}_p$ from~\eqref{eq:Q0Def} for $p \in (p_c(d),1)$. For a given configuration $\mu \in \{0 \in \mathcal{C}_\infty\}$, the canonical law $P_{x,\mu}$ of a simple random walk on $\mathcal{C}_\infty$ starting from $x \in \mathcal{C}_\infty$ is the law of a Markov chain on $\mathbb{Z}^d$ with jump probabilities $q_\mu(x,y)$ for $x,y \in \mathbb{Z}^d$, defined by
\begin{equation}
\begin{split}
q_\mu(x,x) & = 1, \text{ if }\mu_e = 0 \text{ for all } e \in \mathbb{E}^d \text{ with } e \cap \{ x\} \neq \emptyset, \\
q_\mu(x,y) & = \frac{1}{\sum_{y \in \mathbb{Z}^d \, : \, y \sim x} \mu_{\{x,y\}} }, \text{ if } x \sim y \text{ and } \mu_{\{x,y\}} = 1, \\
q_\mu(x,y) & = 0, \text{ otherwise},
\end{split} 
\end{equation}
and we also use the canonical coordinate process $(X_n)_{n \geq 0}$ to denote this random walk. We furthermore define for $n\in \mathbb{N}_0$ the canonical time-shift operator by $n$ units, $\theta_n :(\mathbb{Z}^d)^{\mathbb{N}_0} \rightarrow (\mathbb{Z}^d)^{\mathbb{N}_0}$ which acts on trajectories $\nu \in (\mathbb{Z}^d)^{\mathbb{N}_0}$ via $(\theta_n\nu)(k) = \nu(k+n)$, $k \in \mathbb{N}_0$. \medskip 

We will use the following standard transition probability lower bound from~\cite{barlow2004random} for the random walk on $\mathcal{C}_\infty$. From the latter, one has a set $\Omega_1 \in \mathcal{A}$ with $\mathbb{Q}_p[\Omega_1] = 1$ and random variables $(R_x)_{x \in \mathbb{Z}^d}$ on $(\{0,1\}^{\mathbb{E}^d},\mathcal{A})$ fulfilling
\begin{equation}
\label{eq:R-def}
R_x(\mu) < \infty, \qquad \text{ for }\mu \in \Omega_1, x \in \mathcal{C}_\infty,
\end{equation}
with the property that for $\mu \in \Omega_1$, $x,y \in \mathcal{C}_\infty$, one has a Gaussian lower bound on the transition probability
\begin{equation}
\label{eq:HKLB}
\begin{split}
P_{x,\mu}[X_n = y] +  P_{x,\mu}[X_{n+1} = y] & \geq \frac{c}{n^{\frac{d}{2}}} \exp\left(-\frac{c'|x-y|^2}{n} \right), \qquad \text{for $n \geq R_x(\mu) \vee |x-y|_1$}
\end{split}
\end{equation}
(we use the convention $R_x(\mu) = \infty$ if $\mu \notin \Omega_1$, or $\mu \in \Omega_1$ but $x \notin \mathcal{C}_\infty$). The random variables $R_x$, for $x \in \mathbb{Z}^d$, admit the tail behavior
\begin{equation}
\label{eq:Integrability}
\mathbb{Q}_p[x \in \mathcal{C}_\infty, R_x \geq n] \leq c \exp\left(-n^{\vartheta} \right), \qquad n \in \mathbb{N},
\end{equation}
for some fixed exponent
\begin{equation}
\vartheta = \vartheta(p) > 0,
\end{equation}
see~\cite[(0.5)]{barlow2004random}.

We close this introductory section with an important result concerning the logarithm of the normalized partition function. To that end, recall its definition~\eqref{eq:Normalized-Partition-function} for a fixed realization $\mu \in \{0 \in \mathcal{C}_\infty\}$. By~\cite{liu2009exponential}, see also~\cite[Theorem 1.7]{cosco2021directed} for the straightforward adaptation to the case of general locally finite, connected, infinite graphs, one has for any given $\beta > 0$ a constant $C = C(\beta,\mu)$ such that
\begin{equation}
\label{eq:Concentration-result}
\mathbb{P}\left[\Big\vert \frac{\log W_{n,\mu}^\beta}{n} - \frac{\mathbb{E}[\log W_{n,\mu}^\beta]}{n} \Big\vert \geq \varepsilon \right] \leq \begin{cases}
2e^{-nC\varepsilon^2}, & \text{ for } 0 \leq \varepsilon \leq 1, \\
2e^{-nC\varepsilon}, & \text{ for } \varepsilon > 1.
\end{cases}
\end{equation}

\section{Proof of Theorem~\ref{thm:MainTheorem}, Part (i): Non-randomness of $\overline{\beta}_c(\mathcal{C}_\infty)$}
\label{s.Deterministic}

In this short section, we prove the first part of our main Theorem~\ref{thm:MainTheorem}, which states that the critical inverse temperature $\overline{\beta}_c(\mathcal{C}_\infty)$ is constant for $\widehat{\mathbb{Q}}_p$-almost every realization of the infinite cluster $\mathcal{C}_\infty$. Our argument is based on ergodicity of the family of shift operators under $\mathbb{Q}_p$ in combination with the fact that the definition of the quenched free energy $\mathfrak{f}(\beta,\mu)$ in~\eqref{eq:Quenched-Free-Energy} does not depend on the starting point of the underlying walk on the cluster, see~\cite[Proposition 1.8]{cosco2021directed}.

\begin{proof}[Proof of Theorem~\ref{thm:MainTheorem}, part (i)] 
We first define some auxiliary quantities. Let 
\begin{equation}
\tau_x : \{0,1\}^{\mathbb{E}^d} \rightarrow \{0,1\}^{\mathbb{E}^d}, \qquad  (\tau_x\mu)_e \stackrel{\mathrm{def}}{=} \mu_{e+x}, \qquad x \in \mathbb{Z}^d
\end{equation}
(where for $e = \{y,z\} \in \mathbb{E}^d$, we set $e+x = \{y+x,z+x\}$), denote the canonical shift operator (by $-x$) on the space of percolation configurations. We furthermore define for a fixed $\mu \in \{0 \in \mathcal{C}_\infty\}$, $\omega \in \mathbb{R}^{\mathbb{N} \times \mathbb{Z}^d}$, and general $y \in \mathcal{C}_\infty$, the quantities 
\begin{equation} 
\label{eq:Normalized-Partition-Function-other-starting-point}
\begin{split}
\mathfrak{f}(y,\beta,\mu) & = \limsup_{n \rightarrow \infty} \frac{1}{n} \mathbb{E}\Big[ \log W^{\beta}_{n,\mu}(\cdot,y) \Big], \text{ where} \\
W_{n,\mu}^\beta(\omega,y) &= E_{y,\mu}\left[\exp\left\{\beta \sum_{i = 1}^n \omega(i,X_i) - n\lambda(\beta) \right\}  \right].
\end{split}
\end{equation}
Note that $W_{n,\mu}^\beta(\omega,0) = W_{n,\mu}^\beta(\omega)$ and $\mathfrak{f}(0,\beta,\mu) = \mathfrak{f}(\beta,\mu)$ for any $\mu \in \{0\in \mathcal{C}_\infty\}$, $\omega \in \mathbb{R}^{\mathbb{N} \times \mathbb{Z}^d}$, and $\beta \geq 0$, see~\eqref{eq:Normalized-Partition-function} and~\eqref{eq:Quenched-Free-Energy}. Moreover, we define for $\mu \in \{0 \in \mathcal{C}_\infty\}$ and $y \in \mathcal{C}_\infty$ the quantity 
\begin{equation}
\overline{\beta}_c(\mathcal{C}_\infty,y) = \inf\{\beta \geq 0 \, : \, \mathfrak{f}(y,\beta,\mu) < 0 \}.
\end{equation}
Once $\mu \in \{0 \in \mathcal{C}_\infty\}$ is fixed, one knows by~\cite[Proposition 1.8]{cosco2021directed}  that $\mathfrak{f}(y,\beta,\mu)$ does not depend on $y \in \mathcal{C}_\infty$, and therefore we also see that
\begin{equation}
\label{eq:Ov-beta-constant}
\text{
$\overline{\beta}_c(\mathcal{C}_\infty,y)$ does not depend on $y \in \mathcal{C}_\infty$}.
\end{equation}
We define on the $\mathcal{A}$-measurable event $E = \{\text{there exists a unique infinite cluster $\mathcal{C}_\infty$} \}$ with $\mathbb{Q}_p[E] = 1$ (since $p \in (p_c(d),1)$) the (measurable) function
\begin{equation}
\widehat{x} : E \rightarrow \mathbb{Z}^d, \qquad \widehat{x}(\mu) = \text{ the unique point $y \in \mathcal{C}_\infty$ minimizing $|y|_\infty$},
\end{equation}
with ties broken by using lexicographic order. With these preparations, we define
\begin{equation}
F : \{0,1\}^{\mathbb{E}^d}  \rightarrow [0,\infty), \qquad \mu \mapsto \mathbbm{1}_E(\mu) \overline{\beta}_c(\mathcal{C}_\infty(\mu),\widehat{x}(\mu)) = \begin{cases}
 \overline{\beta}_c(\mathcal{C}_\infty(\mu),\widehat{x}(\mu)), & \text{if }\mu \in E, \\
 0, & \text{if }\mu \notin E.
\end{cases}
\end{equation}
We now argue that 
\begin{equation}
\label{eq:Shift-invariance}
F \circ \tau_x = F, \qquad \text{ for all $x \in \mathbb{Z}^d$}.
\end{equation}
Indeed, suppose that $\mu \notin E$, then $\tau_x \mu \notin E$, and~\eqref{eq:Shift-invariance} holds in this case (with both sides being equal to zero), so we can assume that $\mu \in E$. In the latter case, we have that $\tau_x \mu \in E$ and
\begin{equation}
\begin{split}
F(\tau_x\mu) = \overline{\beta}_c(\mathcal{C}_\infty(\tau_x\mu)),\widehat{x}(\tau_x\mu)) \stackrel{\eqref{eq:Ov-beta-constant}}{=} \overline{\beta}_c(\mathcal{C}_\infty(\mu),\widehat{x}(\mu)) = F(\mu),
\end{split}
\end{equation}
since $\mathcal{C}_\infty(\tau_x\mu)$ is a translate of $\mathcal{C}_\infty(\mu)$, proving~\eqref{eq:Shift-invariance} in this case as well. Since $(\tau_x)_{x \in \mathbb{Z}^d}$ is shift-invariant and ergodic under $\mathbb{Q}_p$, we see that $F$ is $\mathbb{Q}_p$-almost surely constant, and we call this value $\overline{\beta}_c^{\, \mathrm{cluster}}$. Finally, since $\{0 \in \mathcal{C}_\infty\} \subseteq E$, we obtain (using again~\eqref{eq:Ov-beta-constant}) the claim~\eqref{eq:Main-Claim-i}.
\end{proof}

\section{Good open tubes and good blocks in $\mathcal{C}_\infty$}
\label{s.Boxes-percolation}

In this section, we define certain ``favorable regions'' in the infinite cluster of Bernoulli bond percolation which we show to be present at a sufficient, near-volume like number in large enough boxes. Recall assumption~\eqref{eq:Critical-p} and the notation $\mathcal{C}_\infty$ for the percolation cluster below, and define $\theta(p) = \mathbb{Q}_p[0 \in \mathcal{C}_\infty]$ ($ > 0$, by~\eqref{eq:Critical-p}).

\medskip We start by recalling the notion of \textit{good open tubes} from~\cite{nitzschner2025absence}, which will play a prominent role in the proof of a lower bound on the decay rate of $\log W_{n,\mu}^\beta$ when $\beta \in (0,\overline{\beta}_c^{\mathrm{cluster}})$, see Section~\ref{sec:Decay}. To that end, consider the canonical basis $\{e_1,...,e_d\}$ of $\mathbb{R}^d$. We call the set of edges in 
\begin{equation}
T_{x,L} = \Big\{\{x,x+e_1\},\{x+e_1,x+2e_1\},...,\{x+(\lfloor L \rfloor -1)e_1,x+\lfloor L\rfloor e_1\} \Big\}
\end{equation}
 an \textit{open tube} (in direction $e_1$) of length $L \in (0,\infty)$ with base point $x \in \mathbb{Z}^d$ if $\mu_{e} = 1$ for every $e \in T_{x,L}$, and $\mu_e = 0$ for every adjacent edge in coordinate directions perpendicular to $e_1$, with (possibly) the exception of edges $e$ incident to $x$ (i.e.~such that $\{x\} \cap e \neq \varnothing$), and for $e = \{x+\lfloor L\rfloor e_1,x+(\lfloor L \rfloor+1)e_1) \}$  (by convention, $T_{x,L} = \emptyset$ if $\lfloor L \rfloor = 0$). We define the outer (edge) boundary $\partial_{\text{out}}T_{x,L}$ of the tube as the set 
 \begin{equation}
\partial_{\text{out}}T_{x,L} = \Big\{ e \in \mathbb{E}^d \, : \, d_{1}(e, V(T_{x,L}) \setminus \{x\}) = 1 \Big\},
\end{equation} 
namely all edges at $\ell^1$-distance one to $V(T_{x,L}) \setminus \{x\}$ with $V(T_{x,L}) = \{x,x+e_1,...,x+\lfloor L \rfloor e_1\}$ the vertices belonging to the tube. We call an open tube $T_{x,L}$ \textit{good} if all edges in $\partial_{\text{out}}T_{x,L}$ are open. It was proved in~\cite[Lemma 4.1]{nitzschner2025absence} that there are $c(p)n^{d-\delta(\varepsilon)}$ many good open tubes of length $\varepsilon \log(n)$ present with high probability in $B(z,n) \cap \mathcal{C}_\infty$, with $\delta(\varepsilon) \rightarrow 0$ as $\varepsilon \rightarrow 0$, uniformly for all $z \in B(0,n^4)$ and we show in Lemma~\ref{lem:ManyTubes} below a slight generalization of this result adapted to our purposes. \medskip

We now define a second class of favorable regions of $\mathcal{C}_\infty$, which we call \textit{good blocks}, and which can essentially be thought of as pieces of $\mathcal{C}_\infty$ that ``locally look like'' full boxes in $\mathbb{Z}^d$. We show in Lemma~\ref{thm:Many-good-boxes} below that there are $n^{d - \delta(\varepsilon)}$ many good blocks of side length $(\varepsilon \log(n))^{\frac{1}{d}}$ present in $B(0,n) \cap \mathcal{C}_\infty$, where $\delta(\varepsilon) \to 0$ as $\varepsilon \to 0$, with high probability. More precisely, for a real number $L > 0$ and $x \in \mathbb{Z}^d$, we say that $B(x,L)$ is a \textit{good block} if 
\begin{equation}
\mu_e = 1 \qquad \text{for all }e \in E(B(x,L))
\end{equation}
(recall that $E(B(x,L))$ stands for the edges in $\mathbb{E}^d$ with at least one endpoint in $B(x,L)$). We now state the main result of the present section.
\begin{lemma}
\label{thm:Many-good-boxes}
Let $\delta > 0$. There exists $\varepsilon_0(p,\delta) > 0$ such that for $\widehat{\mathbb{Q}}_p$-a.e.~$\mu \in \{0 \in \mathcal{C}_\infty\}$, there is $N_{\mathrm{block}}(\mu,\delta) < \infty$ such that
\begin{equation}
\label{eq:Claim-Good-boxes}
\begin{minipage}{0.8\linewidth}
both $B_o(0,n)$ and $B_e(0,n)$ each contain at least $c(p)n^{d-\delta}$ points $x \in \mathcal{C}_\infty$ such that $B(x,(\varepsilon \log(n))^{\frac{1}{d}})$ is a good block,
\end{minipage}
\end{equation}
for all $n \geq N_{\mathrm{block}}(\mu,\delta)$ and $\varepsilon \in (0,\varepsilon_0)$.
\end{lemma}
\begin{proof}
The proof is an adaptation of~\cite[Lemma 4.1]{nitzschner2025absence}, and we explain the necessary changes. We consider for $n \in \mathbb{N}$ and $\varepsilon > 0$ the relative fraction of points in $B_o(0,n) \cap \mathcal{C}_\infty$ that are centers of good blocks of size $(\varepsilon \log(n))^{\frac{1}{d}}$, i.e.
\begin{equation}
\mathcal{W}_n(\varepsilon) = \frac{1}{|B_o(0,n)|} \sum_{x \in B_o(0,n)} \mathbbm{1}_{\{x \in \mathcal{C}_\infty, B(x,(\varepsilon \log(n))^{\frac{1}{d}}) \text{ is a good block} \}}.
\end{equation}
We then define for $n \in \mathbb{N}$:
\begin{equation}
\theta(n,p,\varepsilon) = \mathbb{Q}_p[x \in \mathcal{C}_\infty, B(x,(\varepsilon \log(n))^{\frac{1}{d}}) \text{ is a good block}].
\end{equation}
We recall that an event $I \in \mathcal{A}$ is called \textit{increasing} if for configurations $\mu,\mu' \in \{0,1\}^{\mathbb{E}^d}$ with $\mu_e \leq \mu_e'$ for all $e \in \mathbb{E}^d$, $\mu \in I$ implies $\mu'\in I$. Since both $\{x \in \mathcal{C}_\infty\}$ and $\{B(x,(\varepsilon \log(n))^{\frac{1}{d}}) \text{ is a good block} \}$ are increasing events, we can directly use the FKG inequality (see, e.g.,~\cite{grimmett1999percolation}) and obtain for any $n \in \mathbb{N}$,
\begin{equation}
\begin{split}
\theta(n,p,\varepsilon) & \geq \mathbb{Q}_p[x \in \mathcal{C}_\infty] \cdot  \mathbb{Q}_p[ B(x,(\varepsilon \log(n))^{\frac{1}{d}}) \text{ is a good block}] \\
& = \theta(p) \cdot p^{|E(B(x,(\varepsilon \log(n))^{\frac{1}{d}} )) |} \\
& \geq \theta(p) \cdot p^{c \lfloor \varepsilon \log(n) \rfloor } \geq \frac{\theta(p)}{n^{\varepsilon c_1(p)}},
\end{split}
\end{equation}
using the fact that $|E(B(x,\ell))| \leq C(2\ell + 1)^d$ for $\ell \in \mathbb{N}$ and $x \in \mathbb{Z}^d$. We center $\mathcal{W}_n(\varepsilon)$ by considering
\begin{equation}
\overline{\mathcal{W}}_n(\varepsilon) = \mathcal{W}_n(\varepsilon) - \theta(n,p,\varepsilon), \qquad n \in \mathbb{N},
\end{equation}
which fulfills $\mathbb{E}_{\mathbb{Q}_p}[\overline{\mathcal{W}}_n(\varepsilon)] = 0$ due to translation invariance of $\mathbb{Q}_p$ (where $\mathbb{E}_{\mathbb{Q}_p}$ denotes the expectation under $\mathbb{Q}_p$). We define $\varepsilon_0 = \varepsilon_0(p,\delta) = \frac{\delta}{c_1(p)}$ and show that
\begin{equation}
\label{eq:Main-Lemma-Bound-to-show}
\mathbb{Q}_p\Big[|\overline{\mathcal{W}}_n(\varepsilon) | \geq \frac{\theta(p)}{2n^{c_1(p)\varepsilon}}  \Big] \leq C\exp\left(-c(p,\delta) n^{c_2(p,\delta)} \right).
\end{equation}
Since the latter is summable over $n \in \mathbb{N}$, an application of the Borel-Cantelli lemma yields the claim~\eqref{eq:Claim-Good-boxes}. To show~\eqref{eq:Main-Lemma-Bound-to-show}, we argue in the same fashion as in~\cite[(4.15)--(4.16)]{nitzschner2025absence} and consider the effect on $\overline{\mathcal{W}}_n(\varepsilon)$ of changing the value of a single edge (this is inspired by a large-deviation type argument in~\cite[Proposition 11]{dario2021quantitative}). To that end, consider
\begin{equation}
\begin{minipage}{0.8\textwidth}
$\widetilde{\mu} = \{\widetilde{\mu}_e \}_{e \in \mathbb{E}^d}$ an i.i.d.~copy of $\mu$, independent of the latter (by working on the canonical space $\{0,1\}^{\mathbb{E}^d} \times \{0,1\}^{\mathbb{E}^d}$ equipped with $\mathbb{Q}_p^{\otimes 2}$, and $(\mu,\widetilde{\mu})$ the canonical projection),
\end{minipage}
\end{equation}
and define for a fixed edge $e \in \mathbb{E}^d$ the resampled configuration $\mu^e_{e'} = \mu_{e'}\mathbbm{1}_{\{e' \neq e \}} + \widetilde{\mu}_{e'}\mathbbm{1}_{\{e' = e \}}$, $e' \in \mathbb{E}^d$. We denote by $\overline{\mathcal{W}}^e_n(\varepsilon)$ the modification of the random variable $\overline{\mathcal{W}}_n(\varepsilon)$ in which $\mu$ is replaced by $\mu^e$, define 
\begin{equation}
\mathcal{C}^e_\infty = \text{ the infinite cluster given configuration $\mu^e$},
\end{equation}
and call $B(x,L)$ an $e$-good block if it is a good block in configuration $\mu^e$.  We now see that
 \begin{equation}
 \label{eq:Deviation-W}
 \begin{split}
 |\overline{\mathcal{W}}^e_n(\varepsilon) - \overline{\mathcal{W}}_n(\varepsilon)| & \leq \frac{1}{|B_o(0,n)|} \bigg\vert \sum_{x \in B_o(0,n)} (\mathbbm{1}_{\{x \in \mathcal{C}_\infty^e, B(x,(\varepsilon \log(n))^{\frac{1}{d}} \text{ is an $e$-good block} \}} \\
 & \qquad -\mathbbm{1}_{\{x \in \mathcal{C}_\infty, B(x,(\varepsilon \log(n))^{\frac{1}{d}} \text{ is a good block} \}} ) \bigg\vert \\
 & \leq \frac{1}{|B_o(0,n)|} | (\mathcal{C}_\infty^e \triangle \mathcal{C}_\infty) \cap B_o(0,n) | + \frac{C\varepsilon \log(n)}{|B_o(0,n)|} \mathbbm{1}_{\{e \in E(B(0,3n))\} },
 \end{split}
 \end{equation}
 (with $\triangle$ denoting the symmetric difference between sets), where we used that 
 \begin{equation}
 \begin{split}
 & |\{x \in B_o(0,n) \, : \, B(x,(\varepsilon \log(n))^{\frac{1}{d}}) \text{ is an $e$-good block} \} \\
& \qquad\triangle \{x \in B_o(0,n) \, : \, B(x,(\varepsilon \log(n))^{\frac{1}{d}}) \text{ is a good block} \}| \leq C\varepsilon\log(n) \mathbbm{1}_{\{e \in E(B(0,3n)) \}},
 \end{split}
 \end{equation}
 as we now explain. Indeed, the only points $x \in B_o(0,n)$ for which the event that $B(x,L)$ is a good block is possibly affected by changing the value of edge $e$ are those for which $d_\infty(x,e) \leq L$, and there are at most $CL^d$ such points. From~\eqref{eq:Deviation-W}, the proof is concluded in the exact same way as in~\cite[(4.17)--(4.25)]{nitzschner2025absence}, and we omit the details (one actually uses~\cite[Lemma A.1]{dario2021quantitative} to calculate an exponential moment of the first summand in the third line of display~\eqref{eq:Deviation-W} and an exponential Efron-Stein-type inequality from~\cite{armstrong2017optimal}).
\end{proof}

We will also need the following refinement of~\cite[Lemma 4.1]{nitzschner2025absence} for tubes. The principal difference compared to~\cite[Lemma 4.1]{nitzschner2025absence} is that the uniformity of the statement is required over the much larger sup-norm ball $B(0,\exp(n^{\xi}))$ instead of $B(0,n)$, where with an exponent $\xi = \xi(p,\delta)$ defined in~\eqref{eq:Exponent}. Informally, this exponent is related to  the stochastic integrability of the relative volume of the cluster $\mathcal{C}_\infty$ in a large box obtained in~\cite[Proposition 11]{dario2021quantitative}, see also Remark~\ref{rem:Final-Remarks}. Again, we only explain the necessary adaptations.
\begin{lemma}
\label{lem:ManyTubes}
Let $\delta > 0$. There exists $\varepsilon_0' = \varepsilon_0'(p,\delta) > 0$ and $\xi = \xi(\delta,p) > 0$ such that for $\widehat{\mathbb{Q}}_p$-a.e.~$\mu \in \{0 \in \mathcal{C}_\infty \}$, there exists $\widetilde{N}_{\textnormal{tube}} = \widetilde{N}_{\textnormal{tube}}(\mu,\delta) < \infty$, such that for all $n \geq \widetilde{N}_{\textnormal{tube}}$, one has:
\begin{equation}
\label{eq:ManyTubes}
\begin{split}
&\textit{For every $z \in B(0,\exp(n^{\xi}))$, the sets $B_o(z,n^{\frac{1}{4}})$ and $B_e(z,n^{\frac{1}{4}})$ each contain at least} \\
& \textit{$c(p) n^{\frac{d-\delta}{4}}$ points $x \in \mathcal{C}_\infty$ for which $T_{x,\varepsilon \log(n)}$ is a good open tube},
 \end{split}
\end{equation}
whenever $\varepsilon \in (0,\frac{1}{4} \varepsilon_0')$.
\end{lemma}
\begin{proof}
Consider random variables for $z \in \mathbb{Z}^d$, defined by
\begin{equation}
\begin{split}
N_z(\mu)  = \inf\{ & n_0 \in \mathbb{N} \, : \, \text{$B_o(z,n)$ contains at least $c_3(p) n^{d-\delta}$ points with $x \in \mathcal{C}_\infty$}  \\
& \text{such that $T_{x,4\varepsilon \log(n)}$ is a good open tube for all $n \geq n_0$} \},
\end{split}
\end{equation}
where $c_3(p) = \frac{1}{2}\theta(p)$. From~\cite[(4.13)]{nitzschner2025absence}, one knows that
\begin{equation}
\mathbb{Q}_p[N_0 > n_0] \leq  C(p,\delta) \exp\left(-c(p,\delta)n_0^{c_4(p,\delta)} \right),
\end{equation}
and the random variables $(N_z)_{z \in \mathbb{Z}^d}$ are equal in law. We define
\begin{equation}
\label{eq:Exponent}
\xi(p,\delta) \stackrel{\mathrm{def}}{=} \frac{c_4(p,\delta)}{8} \wedge 1.
\end{equation}
It follows that \begin{equation}
\sum_{n = 1}^\infty \mathbb{Q}_p\left[ \bigcup_{z \in B(0,\exp(n^{\xi}))} \{N_z > \frac{1}{100} n^{\frac{1}{4}} \}\right] \leq \sum_{n = 1}^\infty C \exp(dn^{\xi}) \cdot \exp\left(-c n^{\frac{c_4}{4}} \right) < \infty,
\end{equation}
we infer (again by the Borel-Cantelli lemma) that for some $N_{\text{tube},o}(\mu) (< \infty\text{, $\mathbb{Q}_p$-a.s.})$, one has that $N_z \leq \frac{1}{100} n^{\frac{1}{4}}$ for every $z \in B(0,\exp(n^\xi))$ whenever $n \geq N_{\text{tube},o}$, which implies that every $B_o(z,k)$ contains at least $c_5(p)k^{d-\delta}$ points $x \in \mathcal{C}_\infty$ such that $T_{x,4\varepsilon \log(k)}$ is a good open tube whenever $k > \frac{1}{100} n^{\frac{1}{4}}$. In particular, every $B_o(z,n^{\frac{1}{4}})$ contains at least $c(p)n^{\frac{d-\delta}{4} }$ many points $x \in \mathcal{C}_\infty$ such that $T_{x,  \varepsilon \log(n)}$ is a good open tube. Since this argument can be repeated for $B_o(\cdot,\cdot)$ replaced by $B_e(\cdot,\cdot)$, we obtain $N_{\text{tube},e}(\mu) (< \infty\text{, $\mathbb{Q}_p$-a.s.})$, and so the claim follows for $n \geq \widetilde{N}_{\text{tube}}(\mu) \stackrel{\text{def}}{=} N_{\text{tube},o}(\mu) \vee N_{\text{tube},e}(\mu)$. 
\end{proof}
\section{Proof of Theorem~\ref{thm:MainTheorem}, Part (ii): $\overline{\beta}_c(\mathcal{C}_\infty) > 0$}
\label{s.ProofMainProp}

In this section, we provide the proof of part (ii) of the main Theorem~\ref{thm:MainTheorem}. Our strategy is to construct, for $n \in \mathbb{N}$ large enough, an event $G_n$ with sub-exponentially decaying probability for the underlying random walk on the cluster (see~\eqref{eq:G_n-event}) to obtain a lower bound on $\log W_{n,\mu}^\beta$ when the trajectories of the underlying walk are ``restricted'' to $G_n$. This is then combined with the concentration result~\eqref{eq:Concentration-result}, which rules out the possibility of very strong disorder, and thus proving the claim. We remark that this strategy is inspired by the one employed in~\cite[Section 6]{cosco2021directed}. There, the authors show that very strong disorder does not hold for small $\beta$ for the directed polymer defined with an upwardly biased random walk on certain supercritical Galton-Watson trees conditioned on non-extinction as a reference model. \medskip

We start our construction by considering a certain beneficial event to which we will restrict the trajectories. To that end, define for a given $\mu \in \{0 \in \mathcal{C}_\infty\}$, $\varepsilon > 0$, and $n \in \mathbb{N}$ the subset of $\mathcal{C}_\infty \cap B_o(0,n^{\frac{1}{4}})$ consisting of centers of good blocks, namely
\begin{equation}
\mathscr{Q}_{o,n} = \{x \in \mathcal{C}_\infty \cap B_o(0,n^{\frac{1}{4}}) \, : \, B(x,(\varepsilon \log(n))^{\frac{1}{d}}) \text{ is a good block} \}
\end{equation}
(note that the latter may be empty), and similarly $\mathscr{Q}_{e,n}$ as the corresponding subset of $\mathcal{C}_\infty \cap B_e(0,n^{\frac{1}{4}})$. We then define the event
\begin{equation}
\label{eq:H_n-def}
H_n = \{X_{\lfloor \sqrt{n} \rfloor} \in \mathscr{Q}_{o,n} \cup \mathscr{Q}_{e,n} \}, \qquad n \in \mathbb{N}.
\end{equation}
Note that the event $H_n$ implicitly depends on the realization $\mathcal{C}_\infty$ of the cluster, but we suppress this dependence in the notation.
\begin{lemma}
\label{lem:H_n-lower-bound}
Let $\delta > 0$. With $\varepsilon_0(p,\delta) > 0$ as in the statement of Lemma~\ref{thm:Many-good-boxes}, there exists a set $\Omega_2 \in \mathcal{A}$ with $\Omega_2 \subseteq \{0 \in \mathcal{C}_\infty\}$ of full $\widehat{\mathbb{Q}}_p$-measure such that for every $\mu \in \Omega_2$, there exists $N_0(\mu)<\infty$ such that for every $n \geq N_0(\mu)$, the following bound holds for $\varepsilon \in (0,4^{-\frac{1}{d}} \varepsilon_0(p,d))$:
\begin{equation}
\label{eq:H_n-lower-bound}
P_{0,\mu}[H_n] \geq \frac{c}{n^{\delta/4}}.
\end{equation}
\end{lemma}
\begin{proof}
Let $\delta > 0$. By Lemma~\ref{thm:Many-good-boxes}, we see that for every $n \geq N^4_{\mathrm{block}}(\mu,\delta)$ as in the statement of that theorem, $|\mathscr{Q}_{o,n}| \geq c(p) n^{\frac{d-\delta}{4}}$ and $|\mathscr{Q}_{e,n}| \geq c(p) n^{\frac{d-\delta}{4}}$, for all $\mu$ in a subset of $\{0 \in \mathcal{C}_\infty\}$ of full $\widehat{\mathbb{Q}}_p$-measure, denoted by $\Omega^\star$. Then, by~\eqref{eq:R-def} we see that for $\mu \in \Omega_1 \cap \Omega^\star (\subseteq \{0 \in \mathcal{C}_\infty\})$, $R_0(\mu) < \infty$ and by the bound~\eqref{eq:HKLB} on the transition probability, one has
\begin{equation}
\begin{split}
P_{0,\mu}[H_n] & \geq  \sum_{x \in \mathscr{Q}_{o,n} \cup \mathscr{Q}_{e,n}} P_{0,\mu}[X_{\lfloor \sqrt{n} \rfloor} = x] \\
& \geq (|\mathscr{Q}_{o,n}| \wedge |\mathscr{Q}_{e,n}|) \cdot \frac{c}{n^{d/4}} \exp\left(- \sup_{x \in B(0,n^{\frac{1}{4}})} \frac{|x|^2}{\sqrt{n}} \right) \\
& \geq \frac{c}{n^{\delta/4}},
\end{split}
\end{equation}
provided that $n \geq N_0(\mu) \stackrel{\mathrm{def}}{=} N^4_{\mathrm{block}}(\mu,\delta) \vee R_0^2(\mu)$, which completes the proof.
\end{proof}
We will combine $H_n$ with the confinement estimate from Lemma~\ref{lem:Confinement} as follows. We define for $r \geq 0$ the time
\begin{equation}
\tau_r = \inf\{k \in \mathbb{N}_0 \, : \, |X_k - X_0|_\infty \geq r \}, 
\end{equation} 
which corresponds to the first time that a random walk (on $\mathcal{C}_\infty$) moves away from its starting point to a sup-norm distance $r$. With this, we define for a given $\mu \in \{0 \in \mathcal{C}_\infty\}$ and $n\in \mathbb{N}$ the event
\begin{equation}
\label{eq:G_n-event}
\begin{split}
G_n & \stackrel{\mathrm{def}}{=} H_n \cap \{\tau_{(\varepsilon \log(n))^{\frac{1}{d}}} \circ \theta_{\lfloor \sqrt{n} \rfloor} \geq n - \lfloor \sqrt{n}\rfloor \} \\
& = \{\text{$X$ is at a center of a good block $B(x,(\varepsilon \log(n))^{\frac{1}{d}})$ at time $\lfloor \sqrt{n} \rfloor$,} \\
& \ \ \text{ and remains in this block until time $n$}\}.
\end{split}
\end{equation}
We will use $G_n$ as a set of good trajectories to which we restrict the normalized partition function. More precisely, we consider for $\mu \in \{0 \in \mathcal{C}_\infty\}$ and $n \in \mathbb{N}$ the random variable
\begin{equation}
\label{eq:Definition-V}
V_{n,\mu} = E_{0,\mu}\left[\exp\left\{\beta\sum_{i = 1}^n \omega(i,X_i) - n\lambda(\beta) \right\} \mathbbm{1}_{G_n}(X) \right].
\end{equation}
We note that by Fubini's theorem and the definition of $\lambda(\beta)$ (see~\eqref{eq:Definitions-of-omega}), it is clear that
\begin{equation}
\label{eq:First-moment}
\mathbb{E}[V_{n,\mu}] = P_{0,\mu}[G_n], \qquad n \in \mathbb{N}.
\end{equation}
We now derive a lower bound on the probability on the right-hand side of the previous display.
\begin{lemma}
Let $\delta > 0$, $\varepsilon \in (0,4^{-\frac{1}{d}} \varepsilon_0(p,d))$ with $\varepsilon_0(p,\delta)$ as in the statement of Lemma~\ref{thm:Many-good-boxes}, and let $\Omega_2 \in \mathcal{A}$, $\Omega_2 \subseteq \{0 \in \mathcal{C}_\infty\}$ be a set of full $\widehat{\mathbb{Q}}_p$-measure and $N_0(\mu) <\infty$ for $\mu \in \Omega_2$ as in Lemma~\ref{lem:H_n-lower-bound}. For every $\mu \in \Omega_2$ and $n \geq N_0(\mu)$, one has
\begin{equation}
\label{eq:G_n-lower-bound}
P_{0,\mu}[G_n] \geq C\exp\left(-c(\varepsilon) \frac{n}{(\log (n))^{\frac{2}{d}}} \right).
\end{equation}
\end{lemma}
\begin{proof}
Let $\mu \in \Omega_2$. By the simple Markov property of $X$ at time $\lfloor \sqrt{n} \rfloor$ we see that
\begin{equation}
\begin{split}
P_{0,\mu}[G_n] = E_{0,\mu}\left[\mathbbm{1}_{H_n} P_{X_{\lfloor \sqrt{n} \rfloor},\mu}\Big[\tau_{(\varepsilon (\log(n))^{\frac{1}{d}}} \geq n - \lfloor \sqrt{n }\rfloor \Big]  \right],
\end{split}
\end{equation}
and we note that on $H_n$, by definition (see~\eqref{eq:H_n-def}) one has $X_{\lfloor \sqrt{n} \rfloor} \in \mathscr{Q}_{o,n} \cup \mathscr{Q}_{e,n}$. For $z \in \mathscr{Q}_{o,n} \cup \mathscr{Q}_{e,n}$ we see that the law of  $X$ under $P_{z,\mu}$ coincides with that under $P^{\mathbb{Z}^d}_z$ until the exit time of $B(z,(\varepsilon \log(n))^{\frac{1}{d}})$, and therefore we have (for $n \geq N_0(\mu)$)
\begin{equation}
\begin{split}
P_{0,\mu}[G_n] & \geq P_{0,\mu}[H_n] \inf_{z \in \mathscr{Q}_{o,n} \cup \mathscr{Q}_{e,n}} P_{z,\mu}[\tau_{(\varepsilon \log(n))^{\frac{1}{d}}} \geq n - \lfloor \sqrt{n} \rfloor] \\
& \stackrel{\eqref{eq:H_n-lower-bound}}{\geq} \frac{c}{n^{\delta/4}} P_0^{\mathbb{Z}^d}[T_{B(0,(\varepsilon \log(n))^{\frac{1}{d}}} \geq n - \lfloor \sqrt{n} \rfloor] \\
& \stackrel{\eqref{eq:Confinement-bound}}{\geq} \frac{c}{n^{\delta/4}} \exp\left(- \frac{n - \lfloor \sqrt{n} \rfloor}{(\varepsilon \log(n))^{\frac{2}{d}}} \right).
\end{split}
\end{equation}
Rearranging terms, we obtain the claim.
\end{proof}
We will use a second moment argument to obtain a lower bound on the probability that $V_{n,\mu}$ is at least $\frac{1}{2}\mathbb{E}[V_{n,\mu}]$ (recall the definition of $V_{n,\mu}$ from~\eqref{eq:Definition-V}). To that end, we use a coupling argument, similarly as in~\cite[Theorem 6.1]{cosco2021directed}. Specifically, we prove the following.
\begin{lemma}
Let $\delta > 0$, $ \varepsilon \in (0,4^{-\frac{1}{d}} \varepsilon_0(p,d))$ with $\varepsilon_0(p,\delta)$ as in the statement of Lemma~\ref{thm:Many-good-boxes},  and let $\Omega_2 \in \mathcal{A}$, $\Omega_2 \subseteq \{0 \in \mathcal{C}_\infty\}$ be a set of full $\widehat{\mathbb{Q}}_p$-measure and $N_0(\mu) <\infty$ for $\mu \in \Omega_2$ as in Lemma~\ref{lem:H_n-lower-bound}. For all $\beta \in (0,\beta_{L^2}(\mathbb{Z}^d))$ and $\mu \in  \Omega_2$, we have for all $n \geq N_0(\mu)$ that
\begin{equation}
\label{eq:Second-moment-upper}
\mathbb{E}[(V_{n,\mu})^2] \leq Ce^{(\lambda(2\beta) - 2\lambda(\beta))\sqrt{n}}.
\end{equation}
\end{lemma}
\begin{proof}
Let $X^{(1)}$ and $X^{(2)}$ two independent walks on $\mathcal{C}_\infty$ under $(P_{0,\mu})^{\otimes 2}$, starting from the origin. To simplify notation, we define for $\beta \geq 0$ the expression
\begin{equation}
\lambda_2(\beta) = \lambda(2\beta) - 2\lambda(\beta).
\end{equation}
We see that
\begin{equation}
\label{eq:Second-moment-calc}
\begin{split}
\mathbb{E}[(V_{n,\mu})^2] & = (E_{0,\mu})^{\otimes 2}\left[\exp\left\{\lambda_2(\beta)\sum_{k = 1}^n \mathbbm{1}_{\{X^{(1)}_k = X^{(2)}_k \}}\right\} \mathbbm{1}_{G_n(X^{(1)})} \cdot \mathbbm{1}_{G_n(X^{(2)})} \right] \\
& \leq e^{\lambda_2(\beta)\sqrt{n}} (E_{0,\mu})^{\otimes 2}\left[\exp\left\{\lambda_2(\beta)\sum_{k = \lfloor \sqrt{n}\rfloor+1}^{n} \mathbbm{1}_{\{X^{(1)}_k = X^{(2)}_k \}}\right\} \mathbbm{1}_{G_n(X^{(1)})} \cdot \mathbbm{1}_{G_n(X^{(2)})} \right].
\end{split}
\end{equation}
We apply the simple Markov property for both walks at time $\lfloor \sqrt{n} \rfloor$, and distinguish two cases:
\begin{itemize}
\item[(i)] Case I: $|X^{(1)}_{\lfloor \sqrt{n} \rfloor} - X^{(2)}_{\lfloor \sqrt{n} \rfloor}|_\infty > 2 (\varepsilon \log(n))^{\frac{1}{d}}$. On the event $G_n(X^{(1)}) \cap G_n(X^{(2)})$, the two walks are then confined in different blocks $B(X^{(1)}_{\lfloor \sqrt{n} \rfloor}, (\varepsilon \log(n))^{\frac{1}{d}})$ and $B(X^{(2)}_{\lfloor \sqrt{n} \rfloor}, (\varepsilon \log(n))^{\frac{1}{d}})$ and do not meet until time $n$.
\item[(ii)] Case II: $X^{(1)}_{\lfloor \sqrt{n} \rfloor} = z$, $X^{(2)}_{\lfloor \sqrt{n} \rfloor} = z'$ with $z,z' \in \mathscr{Q}_{o,n} \cup \mathscr{Q}_{e,n}$ and $|z-z'|_\infty \leq (\varepsilon \log(n))^{\frac{1}{d}}$. We define the random variable
\begin{equation}
Z = \exp\left(\lambda_2(\beta)\sum_{k = 1}^{n - \lfloor \sqrt{n}\rfloor} \mathbbm{1}_{\{X^{(1)}_k = X^{(2)}_k \}} \right)\mathbbm{1}_{\left\{\tau_{(\varepsilon \log(n))^{\frac{1}{d} }}(X^{(1)}) \geq n - \lfloor \sqrt{n} \rfloor \right\}}\mathbbm{1}_{\left\{\tau_{(\varepsilon \log(n))^{\frac{1}{d} }}(X^{(2)}) \geq n - \lfloor \sqrt{n} \rfloor \right\}}.
\end{equation}
In that case note that 
\begin{equation}
\label{eq:Law-domination}
\begin{minipage}{0.8\linewidth}
the law of the random variable $
Z \text{ under $P_{z,\mu} \otimes P_{z',\mu}$}
$ coincides with the law of $Z$ under $P_z^{\mathbb{Z}^d} \otimes P_{z'}^{\mathbb{Z}^d} $,
\end{minipage}
\end{equation}
which follows by a direct inspection of the (finitely many) possible realizations of positive probability of $(X^{(1)}_k,X_k^{(2)})_{k = 1,...,n - \lfloor \sqrt{n} \rfloor}$.
Moreover, we see that the law of $Z$ under $P_z^{\mathbb{Z}^d} \otimes P_{z'}^{\mathbb{Z}^d} $ is stochastically dominated by the law of 
\begin{equation}
Z' = \exp\left(\lambda_2(\beta)\sum_{k = 1}^{\infty} \mathbbm{1}_{\{X^{(1)}_k = X^{(2)}_k \}} \right)
\end{equation}
under $(P_0^{\mathbb{Z}^d})^{\otimes 2}$ (using the strong Markov property at the first collision time and translation invariance). Upon combining this observation with~\eqref{eq:Law-domination}, we see that
\begin{equation}
(E_{z,\mu} \otimes E_{z,\mu'}) [Z] \leq (E_0^{\mathbb{Z}^d})^{\otimes 2}\left[\exp\left\{\lambda_2(\beta) \sum_{k = 1}^\infty \mathbbm{1}_{\{X^{(1)}_k = X^{(2)}_k \}} \right\} \right] \stackrel{\beta < \beta_{L^2}(\mathbb{Z}^d)}{=} C(\beta) < \infty.
\end{equation} 
\end{itemize}
Combining both cases and inserting into~\eqref{eq:Second-moment-calc}, we see that 
\begin{equation}
\mathbb{E}[(V_{n,\mu})^2] \leq e^{\lambda_2(\beta)\sqrt{n}} (1 \vee C(\beta)),
\end{equation}
and the claim follows.
\end{proof}
This brings us in a position to conclude the proof of Theorem~\ref{thm:MainTheorem}.
\begin{proof}[Proof of Theorem~\ref{thm:MainTheorem}, part (ii)]
We consider $\delta > 0$, $\varepsilon_0(p,\delta) > 0$, $ \varepsilon \in (0,4^{-\frac{1}{d}} \varepsilon_0(p,d))$, as well as  $\Omega_2 \in \mathcal{A}$, $\Omega_2 \subseteq \{0 \in \mathcal{C}_\infty\}$  and $N_0(\mu) <\infty$ from Lemma~\ref{lem:H_n-lower-bound}. Upon using the Paley-Zygmund inequality, we see that for $\mu \in \Omega_2$ and $n \geq N_0(\mu) \vee c(\varepsilon,\beta)$, one has
\begin{equation}
\label{eq:After-PZ}
\begin{split}
\mathbb{P}\left[V_{n,\mu} \geq \frac{1}{2} \mathbb{E}[V_{n,\mu}] \right] & \geq \frac{1}{4} \frac{\mathbb{E}[V_{n,\mu}]^2}{\mathbb{E}[(V_{n,\mu})^2]} \\
&\hspace{-0.92cm} \stackrel{\eqref{eq:First-moment},\eqref{eq:G_n-lower-bound},\eqref{eq:Second-moment-upper}}{ \geq} \frac{1}{4} C(\beta)\exp\left(-2c(\varepsilon) \frac{n}{(\log(n))^{\frac{2}{d}}} - \lambda_2(\beta)\sqrt{n}\right) \\
& \geq C'(\beta) \exp\left(-c'(\varepsilon) \frac{n}{(\log(n))^{\frac{2}{d}}}\right).
\end{split}
\end{equation}
Moreover, we have $W^\beta_{n,\mu} \geq V_{n,\mu}$ (recall the definitions~\eqref{eq:Normalized-Partition-function} of $W^\beta_{n,\mu}$ and~\eqref{eq:Definition-V} of $V_{n,\mu}$). We therefore see (using again~\eqref{eq:First-moment} and~\eqref{eq:G_n-lower-bound}) that
\begin{equation}
\label{eq:logW_n-lower}
\mathbb{P}\left[\log W^\beta_{n,\mu} \geq - 2c(\varepsilon)\frac{n}{(\log(n))^{\frac{2}{d}}}  \right] \stackrel{\eqref{eq:After-PZ}} \geq C'(\beta) \exp\left(-c'(\varepsilon)\frac{n}{(\log(n))^{\frac{2}{d}}} \right),
\end{equation}
for every $n \geq N_0(\mu) \vee c(\varepsilon,\beta)$ and $\beta \in (0,\beta_{L^2}(\mathbb{Z}^d))$. As we now see, this is a contradiction to the assumption that very strong disorder holds at $\beta$. Indeed, if very strong disorder holds at $\beta$, we must have
\begin{equation}
\mathbb{E}[\log W^\beta_{n,\mu}] \leq - \eta n
\end{equation}
for $n$ large enough for some fixed $\eta \in (0,1)$, and by~\eqref{eq:Concentration-result} we see
\begin{equation}
\label{eq:logW_n-upper}
\mathbb{P}\left[\log W^\beta_{n,\mu} \geq - \frac{\eta}{2}n \right] \leq e^{-C'(\beta)n\varepsilon^2}.
\end{equation}
We see that for large enough $n \geq N_0(\mu) \vee c(\varepsilon,\beta,\eta)$, there is a contradiction between~\eqref{eq:logW_n-upper} and~\eqref{eq:logW_n-lower}, so very strong disorder cannot hold.
\end{proof}

\section{A bound on the decay of $W_{n,\mu}^\beta$}
\label{sec:Decay}

In this section, we prove Theorem~\ref{thm:Decay-rate}. We introduce for $n \in \mathbb{N}$ the abbreviations
\begin{equation}
\label{eq:r_ndef}
r_n = \left\lfloor\exp(n^\xi)\right\rfloor, \qquad \kappa = \frac{1}{\xi} \ \text{ (recall~\eqref{eq:Exponent})}.
\end{equation} 
We first recall another straightforward modification on a calculation involving fractional moments from~\cite{nitzschner2025absence} that will be useful. 
\begin{lemma}
\label{lem:Main-lemma-Sec5}
Let $\mu \in \{0\in \mathcal{C}_\infty\}$, and recall for $x\in \mathcal{C}_\infty$ the definition of the normalized partition function $W_{n,\mu}^\beta(y)$ starting from $y \in \mathcal{C}_\infty$, see~\eqref{eq:Normalized-Partition-Function-other-starting-point}.
For $\widehat{\mathbb{Q}}_p$-a.s.~$\mu \in \{0 \in \mathcal{C}_\infty\}$, there exists $N'(\mu) <\infty$ such that for all $n \geq N'(\mu)$, one has
\begin{equation}
\label{eq:Unif-over-large-box}
\sup_{y \in \mathcal{C}_\infty \cap B(0,r_n)} \mathbb{E}\left[\sqrt{W_{n,\mu}^\beta(y)} \right] \leq C\exp\left(-c(p)(\log(n))^{\frac{5}{4}}\right).
\end{equation}
\end{lemma}
\begin{proof}
By careful inspection of~\cite[(3.10), (3.20)]{nitzschner2025absence} we see that for any $y\in \mathcal{C}_\infty \cap B(0,r_n)$ (and $n$ large enough),
\begin{equation}
\label{eq:Two-summands}
 \mathbb{E}\left[\sqrt{W_{n,\mu}^\beta(y)} \right] \leq C\exp\left(-c(p)(\log(n))^{\frac{5}{4}}\right) + \sqrt{P_{y,\mu}\left[\mathcal{A}_n^c \right]},
\end{equation}
where $\mathcal{A}_n$ is the event that the random walk $(X_k)_{k = 0}^n$ takes at least $\lfloor \varepsilon \log(n)\rfloor^3$ consecutive steps in an open tube of length $\lfloor \varepsilon \log(n) \rfloor$, and noting that the walk starting from $y \in \mathcal{C}_\infty \cap B(0,r_n)$ stays a box $B(y,n)$ until time $n$. As for the second summand, we note that for any $\eta \in (0,\frac{1}{2})$ there exists $\varepsilon > 0$ small enough such that
\begin{equation}
P_{y,\mu}[\mathcal{A}_n^c] \leq \left(1 - \frac{c(\eta,\mu)}{n^\eta} \right)^{\lfloor \frac{1}{2}\sqrt{n} \rfloor + 1}, \qquad \text{ for }n \geq N_1(\mu,\eta), 
\end{equation}
as in~\cite[(4.35)]{nitzschner2025absence}, where we use Lemma~\ref{lem:ManyTubes} which holds uniformly over $B(0,r_n)\cap \mathcal{C}_\infty$. Since the term in the previous display converges to $0$ at a faster rate than the first summand in~\eqref{eq:Two-summands}, the claim follows.
\end{proof}

We are now in the position to prove the main result of this section, Theorem~\ref{thm:Decay-rate}. Our proof is motivated by an argument in~\cite{comets2005majorizing} which was recently used in~\cite{junk2024strong, junk2025coincidence}. There, the authors use that a relatively weak upper bound on the expectation of a fractional moment of the normalized partition function yields very strong disorder, exploiting translation invariance in an essential way. 

\begin{proof}[Proof of Theorem~\ref{thm:Decay-rate}]
By Jensen's inequality applied for the concave function $f(x) = \sqrt{x}$, $x \geq 0$, we have for any $n \in \mathbb{N}$,
\begin{equation}
\label{eq:Frac-moment}
\mathbb{E}\left[\log W_{n,\mu}^\beta \right] \leq  2  \log \mathbb{E}\left[\sqrt{ W_{n,\mu}^\beta} \right].
\end{equation}
Fix integers $m, n \in \mathbb{N}$ and consider for $x_1,...,x_m \in \mathcal{C}_\infty$ and $y \in \mathcal{C}_\infty$ the random variable
\begin{equation}
\widehat{W}_{mn,\mu}^\beta(y;x_1,...,x_m) = E_{y,\mu}\left[ \exp\left\{\beta \sum_{i = 1}^n \omega(i,X_i) - n\lambda(\beta) \right\} \prod_{i = 1}^m \mathbbm{1}_{\{X_{ni} = x_i \}} \right],
\end{equation}
which is the contribution to $W_{mn,\mu}^\beta(y)$ (defined in~\eqref{eq:Normalized-Partition-Function-other-starting-point}) from underlying paths $(X_k)_{k = 0}^{mn}$ starting from $y \in \mathcal{C}_\infty$ that are at position $x_i$ at times $ni$, $i \in \{1,...,m\}$. We will eventually let $m$ depend on $n$ and tend to infinity at a certain rate. We have
\begin{equation}
W_{mn,\mu}^\beta = \sum_{x_1,...,x_m \in \mathcal{C}_\infty \cap B(0,mn)} \widehat{W}_{mn,\mu}^\beta(0;x_1,...,x_m).
\end{equation}
Using that
\begin{equation}
\label{eq:FracMomentMethod}
\begin{minipage}{0.8\linewidth}
for any finite sequence of positive real numbers $(a_\iota)_{1 \leq \iota \leq N}$ with $N \in \mathbb{N}$, one has $\sqrt{\sum_{\iota = 1}^N a_\iota} \leq \sum_{\iota = 1}^N \sqrt{a_\iota}$,
\end{minipage}
\end{equation}
we find that 
\begin{equation}
\mathbb{E}\left[\sqrt{W_{mn,\mu}^\beta} \right] \leq \sum_{x_1,...,x_m \in \mathcal{C}_\infty \cap B(0,mn)}  \mathbb{E}\left[ \sqrt{\widehat{W}_{mn,\mu}^\beta(0;x_1,...,x_m)} \right].
\end{equation}
By using the simple Markov property for the random walk at time $(m-1)n$, and noting that contributions to the sum above are zero if $|x_{m-1}|_\infty > (m-1)n$, we obtain 
\begin{equation}
\label{eq:Iteration-step-1}
\begin{split}
\mathbb{E}\left[\sqrt{W_{mn,\mu}^\beta} \right] & \leq \sum_{x_1,...,x_{m-1} \in \mathcal{C}_\infty \cap B(0,(m-1)n)}  \mathbb{E}\left[ \sqrt{\widehat{W}_{mn,\mu}^\beta(0;x_1,...,x_{m-1})} \right] \\
& \qquad \times \sum_{x_m \in B(x_{m-1},n)}  \sup_{z \in \mathcal{C}_\infty \cap B(0,(m-1)n)} \mathbb{E}\left[\sqrt{\widehat{W}_{n,\mu}^\beta(z;x_m)}  \right] \\
& \leq \sum_{x_1,...,x_{m-1} \in \mathcal{C}_\infty \cap B(0,(m-1)n)}  \mathbb{E}\left[ \sqrt{\widehat{W}_{mn,\mu}^\beta(0;x_1,...,x_{m-1})} \right] \\
& \qquad \times (2n+1)^d \sup_{z \in \mathcal{C}_\infty \cap B(0,(m-1)n)} \mathbb{E}\left[\sqrt{W_{n,\mu}^\beta(z)}  \right],
\end{split}
\end{equation}
where we used that for any $z \in \mathcal{C}_\infty$,
\begin{equation}
\label{eq:Point-to-point-bound}
\begin{split}
\widehat{W}_{n,\mu}^\beta(z;x_m) & =  E_{z,\mu}\left[ \exp\left\{\beta \sum_{i = 1}^n \omega(i,X_i) - n\lambda(\beta) \right\} \mathbbm{1}_{\{X_{n} = x_m \}} \right]  \\
& \leq E_{z,\mu}\left[ \exp\left\{\beta \sum_{i = 1}^n \omega(i,X_i) - n\lambda(\beta) \right\} \right] = W_{n,\mu}^\beta(z).
\end{split}
\end{equation}
We introduce the shorthand notation
\begin{equation}
U^\beta_{m,n,\mu} = \sup_{z \in \mathcal{C}_\infty \cap B(0,mn)} \mathbb{E}\left[\sqrt{W_{n,\mu}^\beta(z)}  \right], \qquad \text{for $m,n\in \mathbb{N}$}.
\end{equation}
Returning to~\eqref{eq:Iteration-step-1}, we see that
\begin{equation}
\begin{split}
\mathbb{E}\left[\sqrt{W_{mn,\mu}^\beta} \right]  & \leq  \sum_{x_1,...,x_{m-1} \in \mathcal{C}_\infty \cap B(0,(m-1)n)}  \mathbb{E}\left[ \sqrt{\widehat{W}_{mn,\mu}^\beta(0;x_1,...,x_{m-1})} \right] (2n+1)^d U^\beta_{m,n,\mu} \\
& \stackrel{\text{(iterate)}}{\leq} \left((2n+1)^d U_{m,n,\mu}^\beta \right)^m.
\end{split}
\end{equation}
In particular, we see that
\begin{equation}
\frac{\log^\kappa(mn)}{mn} \log \mathbb{E}\left[\sqrt{W^\beta_{mn,\mu}} \right] \leq \frac{\log^\kappa(mn)}{n} \log\left((2n+1)^d U^{\beta}_{m,n,\mu}\right).
\end{equation}
We now choose for a given $n \in \mathbb{N}$ the largest number $m_n \in \mathbb{N}_0$ such that $m_nn \leq r_n$ (recall that $r_n$ is defined in~\eqref{eq:r_ndef}), i.e.
\begin{equation}
\label{eq:mn_def}
m_n = \left\lfloor\frac{r_n}{n}\right\rfloor, \qquad \text{ implying }m_n n \geq r_n - n.
\end{equation}
By~\eqref{eq:Unif-over-large-box} of Lemma~\ref{lem:Main-lemma-Sec5}, we see that $U_{m_n,n,\mu}^\beta \leq C \exp\left(-c(p) (\log(n))^{\frac{5}{4}} \right)$ for $n \geq  N'(\mu) \vee C(p)$ for $\widehat{\mathbb{Q}}_p$-a.e.~$\mu \in \{0 \in \mathcal{C}_\infty\}$, hence
\begin{equation}
\label{eq:Divergence-along-subseq}
\begin{split}
\frac{\log^\kappa(m_n n)}{m_n n} \log \mathbb{E}\left[\sqrt{W^\beta_{m_n n,\mu}} \right] & \leq \frac{\log^\kappa(m_n n)}{n}\left(\widetilde{C} - c'(p)(\log(n))^{\frac{5}{4}} \right) \\
& \stackrel{\eqref{eq:mn_def}}{\leq} \frac{\log^\kappa(r_n - n)}{n}\left(\widetilde{C} - c'(p)(\log(n))^{\frac{5}{4}} \right) \stackrel{\eqref{eq:r_ndef}}{\to} - \infty,
\end{split}
\end{equation}
where the inequality in the second line is valid for large enough $n \geq N'(\mu) \vee C(p)$, and we used that $\log^\kappa(r_n - n) \geq  \log^\kappa(\frac{1}{2}\exp(n^\xi)) \geq \frac{1}{2}n $, again for large enough $n$. \medskip

We have therefore established the divergence to $-\infty$ of $\frac{\log^\kappa(n)}{n} \log \mathbb{E}\left[\sqrt{W_{n,\mu}^\beta} \right]$ along the subsequence where $n$ is replaced by $\widetilde{r}_n = m_nn$ and we now turn to the result on the full sequence. To that end, we first notice that for any $k,\ell \in \mathbb{N}$ and $x_0 \in \mathcal{C}_\infty$, one has
\begin{equation}
W_{k+\ell,\mu}^\beta(x_0) = \sum_{x \in \mathcal{C}_\infty} \widehat{W}_{k,\mu}^\beta(x_0;x) W_{\ell,\mu}^\beta(x) \circ \eta_k,
\end{equation}
using the simple Markov property of the random walk,  
where we used the shift operator $\eta_k$ for $k \in \mathbb{N}_0$, acting on the environment $\omega \in \mathbb{R}^{\mathbb{N} \times \mathbb{Z}^d}$ by
\begin{equation}
(\eta_k \omega)(i,x) = \omega(i+k,x), \qquad i \in \mathbb{N}, x \in \mathbb{Z}^d.
\end{equation}
Since the $(\omega(i,x))_{i \in \mathbb{N},x\in \mathbb{Z}^d}$ are i.i.d.~we obtain (using again~\eqref{eq:FracMomentMethod}) the upper bound
\begin{equation}
\mathbb{E}\left[
\sqrt{W_{k+\ell,\mu}^\beta(x_0)}\right] \leq \sum_{x \in \mathcal{C}_\infty}  \mathbb{E}\left[ \sqrt{\widehat{W}_{k,\mu}^\beta(x_0;x)} \right] \mathbb{E}\left[ \sqrt{ W_{\ell,\mu}^\beta(x)}\right].
\end{equation}
By Jensen's inequality and the fact that $\widehat{W}_{k,\mu}^\beta(x_0;x)$ is zero if $|x-x_0|_\infty > k$, we find, arguing in the same way as in~\eqref{eq:Point-to-point-bound}, that
\begin{equation}
\label{eq:UpperBound-intermediate-k}
\mathbb{E}\left[\sqrt{W_{k+\ell,\mu}^\beta(x_0)}\right]  \leq  C k^d \mathbb{E}\left[ \sqrt{W_{k,\mu}^\beta(x_0)} \right].
\end{equation}
We now consider for any $N \in \mathbb{N}$ the decomposition $N = \widetilde{r}_n + \ell$, $\ell \in [0,...,\widetilde{r}_{n+1} - \widetilde{r}_n)$ with $n = n(N)$ (note that $n(N) \to \infty$ as $N \to \infty$). By~\eqref{eq:Divergence-along-subseq} we know that for any $K > 0$, there exists $n_0$ large enough such that 
\begin{equation}
\frac{\log^\kappa(\widetilde{r}_n)}{\widetilde{r}_n}\log \mathbb{E}\left[\sqrt{W_{\widetilde{r}_n,\mu}^\beta} \right] \leq -K, \qquad \text{ for }n \geq n_0,
\end{equation}
and by~\eqref{eq:UpperBound-intermediate-k} we see that for $N \geq N_0$, 
\begin{equation}
\begin{split}
\frac{\log^\kappa(N)}{N} \log \, & \mathbb{E}\left[\sqrt{W_{N,\mu}^\beta} \right]  \leq \frac{\log^\kappa(N)(C +d\log(N))}{N} - K\frac{\log^\kappa(N) \widetilde{r}_n }{N \log^\kappa(\widetilde{r}_n) } \\
& \leq \frac{\log^\kappa(N)(C +d\log(N))}{N} - K\frac{ \widetilde{r}_n }{\widetilde{r}_{n+1} } \\
& = \frac{\log^\kappa(N)(C +d\log(N))}{N} - K \frac{n}{n+1} \frac{\lfloor \lfloor\exp(n^\xi)\rfloor  /n  \rfloor}{\lfloor \lfloor \exp((n+1)^{\xi}) \rfloor /(n+1) \rfloor}.
\end{split}
\end{equation}
Now letting $N$ tend to infinity, we see that 
\begin{equation}
\limsup_{N \rightarrow \infty} \frac{\log^\kappa(N)}{N} \log \mathbb{E}\left[\sqrt{W_{N,\mu}^\beta} \right] \leq  -\frac{K}{e}.
\end{equation}
Since $K > 0$ was arbitrary, this shows the claim. Finally, we use again~\eqref{eq:Concentration-result} to see that for any $\varepsilon > 0$, for large $n$,
\begin{equation}
\mathbb{P}\left[ \log^\kappa(n) \Big\vert \frac{\log W_{n,\mu}^\beta}{n}  - \frac{\mathbb{E}[\log W_{n,\mu}^\beta]}{n} \Big\vert \geq \varepsilon \right] \leq 2 \exp\left( - C \varepsilon^2 \frac{n}{\log^{2\kappa}(n)} \right).
\end{equation}
It follows by the Borel-Cantelli lemma that $\mathbb{P}$-a.s.~for large enough $n$ (depending on the environment $\omega$ and on $\mu$), 
\begin{equation}
\log^\kappa(n) \frac{\log W_{n,\mu}^\beta}{n} \leq \log^\kappa(n) \frac{\mathbb{E}[\log W_{n,\mu}^\beta]}{n} + \varepsilon,
\end{equation}
and since the right-hand side diverges to $-\infty$, so does the left-hand side. 
\end{proof}

\begin{remark}
\label{rem:Final-Remarks}
\begin{enumerate}
\item[(1)] Our main result~\eqref{eq:Main-result-bound} shows that very strong disorder does not hold for the directed polymer on $\mathcal{C}_\infty$ for $\widehat{\mathbb{Q}}_p$-a.e.~$\mu \in \{0 \in \mathcal{C}_\infty\}$ whenever $\beta \in (0, \beta_{L^2}(\mathbb{Z}^d))$ for $d \geq 3$. However, by combining the results of~\cite{berger2010critical,birkner2011collision,
birkner,
birkner2010annealed,birkner2011disorder} (see also~\cite[Theorem B]{junk2024strong} and~\cite[Section 3.2]{zygouras2024directed}), one knows that $\beta_{L^2}(\mathbb{Z}^d) < \beta_c(\mathbb{Z}^d) ( = \overline{\beta}_c(\mathbb{Z}^d))$ for $d \geq 3$. One may naturally wonder how $\overline{\beta}_c^{\, \mathrm{cluster}}$ compares to $\overline{\beta}_c(\mathbb{Z}^d)$.
\item[(2)] Our bound on the decay rate obtained in~\eqref{eq:Bound-main-result} is rather rough, being defined in~\eqref{eq:Exponent} in terms of the implicit $c_4(p,\delta)$, which corresponds to the same quantity in~\cite[(4.9)]{nitzschner2025absence} (see also Remark 4.2 in the latter), and ultimately on the stochastic integrability obtained in~\cite[Proposition 11, Remark 16]{dario2021quantitative}, which is not optimal. To make progress on the precise rate would require a precise control of atypical sets attached to the cluster, which is out of the scope of the methods presented here. We also remark that sets of bad isoperimetry such as tubes in $\mathcal{C}_\infty$ appear in other contexts, including effective resistances or the maximum of the Gaussian free field on the cluster, see~\cite{abe2015effective,
schweiger2024maximum}.
\item[(3)] Since $\beta_c(\mathcal{C}_\infty) = 0$ for $\widehat{\mathbb{Q}}_p$-a.e.~$\mu \in \{0 \in \mathcal{C}_\infty\}$, one might also wonder in a different direction whether a phase transition can be recovered by ``zooming in'' around $\beta_c = 0$ by rescaling the inverse temperature $\beta = \beta_n \to 0$ (as $n \to \infty$). For the corresponding model on $\mathbb{Z}^d$, this strategy leads to an \textit{intermediate disorder regime}, which was studied in detail in~\cite{AKQ-14} in $d = 1$ and in~\cite{CSZ-AAP17} in $d = 2$. In particular, in $d = 2$ a precise picture of a phase transition in $\widehat{\beta}$ occurring when $\beta \sim \frac{\widehat{\beta}}{\sqrt{\log n}}$ has emerged, see, e.g.,~\cite{CCR25,CSZ-CMP19,CSZ23,GQT-21,tsai2024stochastic} and references therein for results at or near criticality.  
\end{enumerate}
\end{remark}

\subsection*{Acknowledgements}
Most of the research for the present paper was carried out while FC was supported by the Luxembourg National Research Fund (AFR/22/17170047/Bilateral-GRAALS). FC also acknowledges support from the European Union's Horizon 2020 research and innovation programme under the Marie Skłodowska-Curie grant agreement No 101034255. FC acknowledges the support of INdAM/GNAMPA. MN was partially supported by Hong Kong RGC grants ECS 26301824 and GRF 16303825.  The authors are grateful to Hubert Lacoin for suggesting an approach to obtain very strong disorder that inspired the proof presented in Section~\ref{s.ProofMainProp} and to Alberto Chiarini for useful discussions that lead to a simplification of the proof in Section~\ref{s.Deterministic}.

\bibliographystyle{abbrv}
\bibliography{PolymersRef}

\end{document}